\documentclass{amsart}
\usepackage{hyperref}
\usepackage{amsmath,amsthm,amssymb,comment,enumerate}
\usepackage{graphicx}
\usepackage{color}

\usepackage{subcaption}
\usepackage{float}
\usepackage{comment}

\usepackage{todonotes}
\theoremstyle{plain}
\newtheorem{thm}{Theorem}[section]

\newtheorem{lem}[thm]{Lemma}
\newtheorem{prop}[thm]{Proposition}

\theoremstyle{definition}

\theoremstyle{remark}
\newtheorem{rem}[thm]{Remark}

\numberwithin{equation}{section}

\begin{document}
\title{The first Grushin eigenvalue on cartesian product domains}%
\author{Paolo Luzzini}%
\author{Luigi Provenzano}%
\author{Joachim Stubbe}%

\address{Paolo Luzzini, Dipartimento di  Matematica `Tullio Levi-Civita', Universit\`a degli Studi di Padova, Italy {\tt pluzzini@math.unipd.it}}%
\address{Luigi Provenzano,  SBAI Sapienza Universit\`a di Roma, Via Antonio Scarpa 16, 00161 Roma, Italy {\tt luigi.provenzano@uniroma1.it}}
\address{Joachim Stubbe, EPFL, SB MATH SCI-SB-JS, Station 8, CH-1015 Lausanne, Switzerland {\tt Joachim.Stubbe@epfl.ch}}%


\date{\today}%
\begin{abstract} 
In this paper we consider the first eigenvalue $\lambda_1(\Omega)$ of the Grushin operator $\Delta_G:=\Delta_{x_1}+|x_1|^{2s}\Delta_{x_2}$ with Dirichlet boundary conditions on a bounded domain $\Omega$ of $\mathbb{R}^d= \mathbb{R}^{d_1+d_2}$. We prove that $\lambda_1(\Omega)$ admits a unique minimizer 
in the class of domains with prescribed finite volume which are the cartesian product of a set in $\mathbb{R}^{d_1}$ and a set in $\mathbb{R}^{d_2}$, and that the minimizer is the product of two balls $\Omega^*_1 \subseteq \mathbb{R}^{d_1}$ and 
$\Omega_2^* \subseteq \mathbb{R}^{d_2}$.  Moreover, we provide a lower bound for $|\Omega^*_1|$ and for 
$\lambda_1(\Omega_1^*\times\Omega_2^*)$.  Finally, we consider the limiting problem as $s$ tends to $0$ and to $+\infty$.



\end{abstract}
\maketitle
\setcounter{page}{1}
\tableofcontents


\section{Introduction}

We consider the Grushin operator $\Delta_G$ in $\mathbb{R}^d$  defined by  
\begin{equation*}\label{Gdef}
\Delta_{G} := \Delta_{x_1} +|x_1|^{2s}\Delta_{x_2} \qquad s >0,
\end{equation*}
where  $d \in \mathbb{N}$, $d \geq 2$, $d_1, d_2 \in \mathbb{N}$, 
$d= d_1+d_2$, $x_1 \in \mathbb{R}^{d_1}$, $x_2 \in \mathbb{R}^{d_2}$. Here $x_1$ and $x_2$  denote the first $d_1$ and last   $d_2$ components of $x \in \mathbb{R}^d$ and $\Delta_{x_i}$   denotes the standard Laplacian with respect to  $x_i$, $i=1,2$.   

As one can immediately  realize, $\Delta_G$ is not uniformly elliptic since it degenerates to $\Delta_{x_1}$ on the 
$x_1$-axis. In addition,  if $s \in \mathbb{N}$, it can be written as
\[
\Delta_G = \sum_{i=1}^k X^2_i
\]
where $k \in \mathbb{N}$ and $\{X_i\}_{i=1,\ldots,k}$ is a family of smooth vector fields satisfying the H\"ormander 
condition,  i.e. $\{X_i\}_{i=1,\ldots,k}$ generates a Lie algebra  of maximum
rank at any point  (see H\"ormander \cite{Ho67}). However in general (i.e. for $s \not \in \mathbb{N}$) the H\"ormander condition fails to hold since the generating vector fields are not smooth. 

The operator $\Delta_G$ has been independently introduced  by Baouendi \cite{Ba67} and Grushin  \cite{Gr70, Gr71}. Later on, it has been generalized and further studied by several authors under different points of view. Here we mention, without the sake of completeness, Franchi and Lanconelli \cite{FrLa82, FrLa83, FrLa84} for the H\"older regularity of weak solutions and for the embedding of the associated Sobolev spaces, Garofalo and Shen \cite{GaSh94} for Carleman estimates and unique continuation results, D'Ambrosio \cite{Da04} for Hardy inequalities, Thuy and Tri \cite{ThTr12} and Kogoj and Lanconelli \cite{KoLa12} for semilinear problems. Finally we mention  Chen and Chen \cite{ChCh20}, Chen, Chen, Duan and Xu \cite{ChChDu17}, Chen, Chen and Li \cite{ChChLi20} and Chen and Luo \cite{ChLu15} for asymptotic bounds for eigenvalues.

It is well known that the spectrum of the problem
   \begin{equation}\label{Geigp}
\begin{cases}
-\Delta_{G} u = \lambda  u \qquad &\mbox{ in } \Omega,\\
u=0 \qquad &\mbox{ on } \partial \Omega, 
\end{cases}
\end{equation}
in a bounded domain (i.e. a connected open set)  $\Omega$ of $\mathbb{R}^d$ is made of eigenvalues of finite multiplicity that can be arranged in a divergent sequence:
\[
0 < \lambda_1(\Omega) \leq \cdots \leq \lambda_j(\Omega) \leq \cdots \nearrow +\infty.
\]
In the present paper we are interested in moving some steps toward the understanding of the minimization problem of the first eigenvalue $\lambda_1(\Omega)$ among domains $\Omega$ with prescribed finite volume.  Since the seminal works of Faber \cite{Fa23} and Krahn  \cite{Kr24}, it is known that  the ball minimizes the first eigenvalue of the Dirichlet Laplacian among all the domains with a  fixed volume (see also Henrot \cite{He06} for a monograph on optimization problems for eigenvalues of elliptic operators). The same problem for degenerate operators is far from being understood and, to the best
of our  knowledge, no conclusive results for the optimization of Grushin eigenvalues are available in the literature, not even for the minimization of the first eigenvalue. In particular, an optimal shape for the first eigenvalue is not even conjectured, neither in the simplest case  $d=2$, $s=1$, and in general it is not an euclidean ball (see Section \ref{sec:num})

It is worth mentioning that in \cite{LaMuLu21} the authors showed that, when $s  \in \mathbb{N}$, the symmetric functions of the eigenvalues 
$\{\lambda_n\}_{n \in \mathbb{N}}$ depend real analytically upon suitable perturbations of the domain and proved an explicit Hadamard-type formula for their shape differential. This formula is then used to characterize critical domains under isovolumetric  perturbations via an overdetermined problem, which for the first eigenvalue $\lambda_1$ with normalized eigenfunction $u_1$ consists of finding the domains such that the following problem is satisfied:
\begin{equation}
\begin{cases}
-\Delta_G u_1 =\lambda_1u_1 \qquad &\mbox{ in } \Omega,\\
u_1=0\qquad &\mbox{ on } \partial\Omega,\\
\nabla u_1 \cdot(\nu_{x_1}, |x_1|^s  \nu_{x_2})=\mbox{const.}\qquad &\mbox{ on } \partial\Omega.
\end{cases}
\end{equation} 
Here above, $\nu=(\nu_{x_1},\nu_{x_2})$ denotes the outer unit normal field to $\partial \Omega$.
To the best of our knowledge, the understanding of this kind of overdetermined problems for degenerate operators is at the moment limited and thus no information on critical domains can be extracted from them.

Our point of view in order to give a first and  partial  answer to the problem of minimizing the first 
Grushin eigenvalue is to consider the case in which  
  $\Omega$ is 
 the cartesian product of two bounded domains $\Omega_1 \subseteq \mathbb{R}^{d_1}$, $\Omega_2 \subseteq \mathbb{R}^{d_2}$. That is for $V>0$ fixed, we set
 \begin{equation*}
  \mathcal{A}(V) := \Big\{\Omega_1 \times \Omega_2 : \Omega_1 \subseteq \mathbb{R}^{d_1}, \Omega_2 \subseteq \mathbb{R}^{d_2},  \Omega_1,\Omega_2\mbox{ bounded domains, } |\Omega_1||\Omega_2|=V\Big\},
 \end{equation*}
 and we consider the minimization problem
 \begin{equation}\label{fkprob}
 \min_{\Omega \in \mathcal{A}(V)}\lambda_1(\Omega).
 \end{equation}
 By separation of variables, problem \eqref{Geigp} decouples into two problems. The first one is a problem for the 
 standard Laplacian in $\mathbb{R}^{d_2}$ and the second one for the Schr\"odinger operator with potential 
 $\mu |x_1|^{2s}$  in $\mathbb{R}^{d_1}$, where $\mu$ is the coupling constant.
 Our main result shows that problem \eqref{fkprob} admits a unique minimizer which is the product of two balls
 $\Omega_1^*$ and $\Omega_2^*$
 in $\mathbb{R}^{d_1}$ and $\mathbb{R}^{d_2}$, respectively (see Theorem \ref{thm:uniqmin}). 
 The main tool of the uniqueness proof relies on a differential inequality involving the second derivative of the
  first Schr\"odinger eigenvalue with respect to the coupling constant (see Proposition \ref{prop:pos}).
 As a further result, we provide some information on the localization of this unique minimum by proving a lower bound for $|\Omega_1^*|$, which in turn implies a lower bound for $\lambda_1(\Omega_1^*\times\Omega_2^*)$
 (see Propositions \ref{prop:lb} and \ref{prop:lbe}). Then, we study the asymptotic behavior of the problem when $s \to 0$ and $s \to +\infty$ and we deduce that our lower bounds are sharp in these limits.
Finally, we provide some numerical computations in the planar case, that is for 
$d_1=d_2=1$. We first numerically solve the minimization problem for some value of $s>0$ and then we also compute the first eigenvalue in the case of balls in $\mathbb{R}^2$ and we compare it with the first eigenvalue on rectangles.

The paper is organized as follows: Section \ref{sec:pre} contains some preliminaries on the eigenvalue problem for the 
Grushin operator $\Delta_G$.   In Section \ref{sec:min} we prove our main results on the minimization problem on cartesian product domains. In particular we prove that the minimization problem for the first Grushin eigenvalue 
admits a unique minimum, we provide some information on the localization of this minimum proving a lower bound, and 
we study the behavior of the problem when $s \to 0$ and $s \to +\infty$. Finally, in Section \ref{sec:num} we present 
the numerical computations.


\section{Preliminaries on the eigenvalue problem}\label{sec:pre}
  Let $\Omega$ be a bounded domain in 
$\mathbb{R}^d$. We retain the standard notation for the Lebesgue space $L^2(\Omega)$ of real-valued square integrable functions.
We denote by 
$H_G^{1}(\Omega)$  the space of functions in $L^2(\Omega)$ such that 
$\nabla_{x_1}u\in (L^2(\Omega))^{d_1}$  and 
 $ |x_1|^s\nabla_{x_2}u\in (L^2(\Omega))^{d_2}$.
The space $H_G^{1}(\Omega)$ is a Hilbert space with the following scalar product:
\begin{align*}
\langle u,v\rangle_{G,2} : = \langle u,v\rangle_{2} 
+\langle \nabla_{x_1}u,\nabla_{x_1} v\rangle_{2}
 +\langle  |x_1|^{s}\nabla_{x_2}u,  |x_1|^s\nabla_{x_2} v\rangle_{2} \quad \forall u,v \in H_G^{1}(\Omega).
\end{align*}
Here $ \langle \cdot,\cdot\rangle_{2}$ denotes the 
standard scalar product in $L^2(\Omega)$.  
Moreover, if $u \in H_G^{1}(\Omega)$ we set
\begin{equation}\label{defnablaG}
\nabla_G u : = \big(\nabla_{x_1}u, |x_1|^s\nabla_{x_2}u\big)
\end{equation}
and we refer to $\nabla_G u$ as the Grushin gradient of $u$. 
We denote by $H_{G,0}^{1}(\Omega)$ the closure of $C_c^\infty(\Omega)$ in $H_G^{1}(\Omega)$. 
Analogs of the Rellich-Kondrachov embedding theorem and of the Poincar\'e inequality hold in  $H_{G,0}^{1}(\Omega)$.
 That is, the following theorems hold (for a 
 proof we refer to Franchi and Serapioni \cite[Thm. 4.6]{FrSe87} and to D'Ambrosio \cite[Thm. 3.7]{Da04}, respectively).

  \begin{thm}[Rellich-Kondrachov]\label{RK}
Let $\Omega$ be a bounded domain in $\mathbb{R}^d$.  Then the space $H_{G,0}^{1}(\Omega)$ 
is compactly embedded in $L^2(\Omega)$.
\end{thm}

\begin{thm}[Poincar\'e inequality]\label{PI}
Let $\Omega$ be a bounded domain in $\mathbb{R}^d$.  Then there exists $C>0$ such that
\begin{equation*}
\|u\|_{L^2(\Omega)} \leq C \|\nabla_G u\|_{(L^2(\Omega))^d} \qquad \forall u \in H_{G,0}^{1}(\Omega).
\end{equation*}
\end{thm}
We consider the eigenvalue problem for the Grushin operator with Dirichlet boundary conditions: 
\begin{equation}\label{ssp}
\begin{cases}
-\Delta_{G} u = \lambda  u \qquad &\mbox{ in } \Omega,\\
u=0 \qquad &\mbox{ on } \partial \Omega,
\end{cases}
\end{equation}
in the unknowns $\lambda$ (the eigenvalue) and $u$ (the eigenfunction). Problem \eqref{ssp} is understood in the weak sense as follows:   
\begin{equation}\label{wsp}
\int_{\Omega} \nabla_Gu \cdot \nabla_Gv\,dx  = \lambda \int_{\Omega} uv\,dx \qquad \forall v \in  H_{G,0}^{1}(\Omega)
\end{equation}
in the unknowns $\lambda \in \mathbb{R}$ and $u \in H_{G,0}^{1}(\Omega)$.
By  Theorem \ref{RK}, Theorem \ref{PI} and by a standard procedure in spectral theory, 
problem \eqref{wsp} can be recast as an eigenvalue problem for a compact self-adjoint operator in $L^2(\Omega)$.
In particular, the eigenvalues of equation \eqref{wsp} have finite multiplicity and can be represented by means of a divergent sequence:
\[ 
0 < \lambda_1(\Omega) \leq \lambda_2(\Omega) \leq \cdots \leq\lambda_j(\Omega)\leq \cdots \nearrow +\infty.
\]
Moreover,  by the min-max principle (see  Davies \cite[\S 4.5]{Da95}),  the following variational characterization holds:
\begin{equation*}\label{min-max}
\lambda_j(\Omega) = \min_{\substack{E \subseteq H_{G,0}^{1}(\Omega)\\ \mathrm{dim} E = j}} 
\max_{\substack{u \in E \\ u\neq 0}} \frac{\int_{\Omega}|\nabla_G u|^2\,dx}{\int_{\Omega} u^2 \, dx} 
\qquad \forall j \in \mathbb{N}.
\end{equation*}
We note that by Monticelli and Payne \cite[Thm. 6.4]{MoPa09} there exists a non-negative eigenfunction $u_1$ corresponding to the first eigenvalue $\lambda_1(\Omega)$. In addition,  $\lambda_1(\Omega)$ is known to be simple 
if $s \in \mathbb{N}$ and $\Omega$ is connected and non-characteristic (see Chen and Chen \cite[Prop. A.2]{ChCh20})
or if $\Omega \setminus \{x_1=0\}$ is connected (see Monticelli and Payne \cite[Thm. 6.4]{MoPa09}).



\section{The eigenvalue problem in cartesian product domains}\label{sec:min}
Here we consider the eigenvalue problem for the Dirichlet Grushin operator \eqref{ssp} in cartesian product domains, so that it is possible to proceed by separation of variables. 
Let $\Omega_1 \subseteq \mathbb{R}^{d_1}$, $\Omega_2 \subseteq \mathbb{R}^{d_2}$ be two bounded  domains and let $\Omega= \Omega_1 \times \Omega_2$. We claim that the solutions $u$ of problem \eqref{ssp} can be written as
\[
u(x_1,x_2) = f(x_1)g(x_2) \qquad   (x_1,x_2) \in \Omega_1 \times \Omega_2.
\]
In this case, \eqref{ssp} becomes
\[
-g(x_2)\Delta_{x_1} f(x_1) - |x_1|^{2s}f(x_1)\Delta_{x_2} g(x_2) = \lambda f(x_1) g (x_2) \qquad  (x_1,x_2) \in  \Omega_1 \times \Omega_2,
\]
 which is equivalent to 
\[
g(x_2) \left( -\Delta_{x_1} f(x_1) -\lambda f(x_1)\right) = |x_1|^{2s} f(x_1) \Delta_{x_2} g(x_2) \qquad  (x_1,x_2) \in \Omega_1 \times \Omega_2.
\]
Separating the equations and imposing the boundary conditions we get that for some $\mu >0$ one has
\begin{equation}\label{eigsplit1}
\begin{cases}
-\Delta_{x_2} g = \mu g \qquad &\mbox{ in } \Omega_2,\\
g = 0\qquad &\mbox{ on } \partial\Omega_2,
\end{cases}
\end{equation}
and 
\begin{equation}\label{eigsplit2}
\begin{cases}
-\Delta_{x_1} f + \mu |x_1|^{2s}f  = \lambda f \qquad &\mbox{ in } \Omega_1,\\
f = 0\qquad &\mbox{ on } \partial\Omega_1.
\end{cases}
\end{equation}
The eigenvalue problem   is then splitted into two coupled eigenvalue problems, one 
for the Laplacian and the other for the Schr\"odinger operator with potential $\mu |x_1|^{2s}$. As it is well-known, problem \eqref{eigsplit1} admits a sequence of eigenvalues
\[
0 < \mu_1(\Omega_2) < \mu_2(\Omega_2) \leq \cdots \leq \mu_j(\Omega_2) \leq \cdots \nearrow +\infty,
\]
with corresponding eigenfunctions $\{g_j\}_{j \in \mathbb{N}}$ orthonormal in $L^2(\Omega_2)$, whereas problem \eqref{eigsplit2}, for each fixed $\mu>0$, admits a sequence of eigenvalues
\[
0 <E_{1}(\mu,\Omega_1) < E_{2}(\mu, \Omega_1) \leq \cdots \leq E_{j}(\mu, \Omega_1) \leq \cdots \nearrow +\infty,
\]
with eigenfunctions $\{f^\mu_j\}_{j \in \mathbb{N}}$ orthonormal in $L^2(\Omega_1)$.
We note that by the min-max principle, the first eigenvalue $E_{1}(\mu,\Omega_1)$ of problem \eqref{eigsplit2}  is given by 
\[
E_{1}(\mu,\Omega_1) = \min_{f \in H^1_0(\Omega_1)\setminus \{0\}}\frac{\int_{\Omega_1}|\nabla_{x_1}f|^2+\mu|x_1|^{2s}f^2\,dx_1}{\int_{\Omega_1}f^2\,dx_1}.
\] 
Here above, and throughout the paper, $H_0^1(\Omega_1)$ denotes the closure of $C_c^{\infty}(\Omega_1)$ 
with respect to the norm $\left(\|f\|_2^2+\|\nabla_{x_1} f\|_2^2\right)^{\frac{1}{2}}$ of  $H^1(\Omega_1)$. 
Therefore, a family of eigenvalues is given by $\{E_{j}\big(\mu_k(\Omega_2),\Omega_1\big)\}_{j,k \in \mathbb{N}}$ with associated eigenfunctions  $\left\{g_kf^{\mu_k(\Omega_2)}_j\right\}_{j,k \in \mathbb{N}}$.
The claim is proved since we observe   that $\left\{g_kf^{\mu_k(\Omega_2)}_j\right\}_{j,k \in \mathbb{N}}$ is a complete system 
in $L^2(\Omega_1\times\Omega_2)$, and then
\[
\{\lambda_n(\Omega)\}_{n \in \mathbb{N}}  = \{E_{j}\big(\mu_k(\Omega_2),\Omega_1\big)\}_{j,k \in \mathbb{N}} 
\]
and 
\begin{equation}\label{mml1}
\lambda_1(\Omega) = E_{1}\big(\mu_1(\Omega_2),\Omega_1\big) = \min_{f \in H^1_0(\Omega_1) \setminus \{0\}}\frac{\int_{\Omega_1}|\nabla_{x_1}f|^2+\mu_1(\Omega_2)|x_1|^{2s}f^2\,dx_1}{\int_{\Omega_1}f^2\,dx_1}.
\end{equation}
\begin{rem}
 Since the first eigenvalue of the Dirichlet Laplacian $-\Delta_{x_2}$ on $\Omega_2$ and the first eigenvalue of the Dirichlet Schr\"odinger  operator
$-\Delta_{x_1} + \mu |x_1|^{2s}$
 on $\Omega_1$ are well-known to be simple, it is immediately seen that in the case of a cartesian product domain 
 $\Omega = \Omega_1 \times \Omega_2$, $\Omega_1 \subseteq \mathbb{R}^{d_1}$ and  
 $\Omega_2 \subseteq \mathbb{R}^{d_2}$,
  the first eigenvalue  $\lambda_1(\Omega)$ 
is  simple without requiring any additional assumption. However, as already pointed out, the simplicity of the first Grushin eigenvalue  it is known to hold without requiring $\Omega$ to be a cartesian product domain  under some additional assumptions (see Monticelli and Payne
  \cite[Thm. 6.4]{MoPa09}, Chen and Chen \cite[Prop. A.2]{ChCh20}).
\end{rem}

\subsection{Existence of a minimum}\label{sec:minprod}

 We now start to consider the minimization problem for the first eigenvalue in cartesian product domains with a prescribed  volume.
 That is we fix $V>0$ and we consider the minimization problem   
  \begin{equation}\label{minpb}
 \min_{\Omega \in \mathcal{A}(V)}\lambda_1(\Omega),
 \end{equation}
where  
  \begin{equation*}
  \mathcal{A}(V) = \Big\{\Omega_1 \times \Omega_2 : \Omega_1 \subseteq \mathbb{R}^{d_1}, \Omega_2 \subseteq \mathbb{R}^{d_2}, \Omega_1,\Omega_2\mbox{ bounded domains, } |\Omega_1||\Omega_2|=V\Big\}.
 \end{equation*}

  \begin{rem}
  As it is well-known, if one removes a  zero capacity set from either $\Omega_1$ or $\Omega_2$, the eigenvalues of 
  problems \eqref{eigsplit1} and \eqref{eigsplit2} remain the same. Thus here in this paper we do not allow this kind of irregularity in the domains and when speaking of uniqueness of minimizers we always mean uniqueness up to sets of zero capacity (see also Henrot \cite[\S 3.2]{He06}).
  \end{rem}
  \begin{rem}
Instead of considering the minimization problem \eqref{minpb}, 
one  can for instance consider the more simple problem of minimizing $\lambda_1(\Omega)$  in the class of cartesian products with each product domain having prescribed volume, that is in 
\begin{align*}
\mathcal{B}(V_1,V_2) = \Big\{\Omega_1 \times \Omega_2 : \,\,&\Omega_1 \subseteq \mathbb{R}^{d_1}, \Omega_2 \subseteq \mathbb{R}^{d_2}, \Omega_1,\,\Omega_2\mbox{ bounded domains, } \\
&|\Omega_1|=V_1,|\Omega_2|=V_2\Big\}
\end{align*}
for some $V_1,V_2 >0$.
Since the ball in $\mathbb{R}^{d_2}$ with volume $V_2$, which we denote by $B_2(V_2)$,  is the unique minimizer of the first eigenvalue of the Dirichlet Laplacian among all the domains with volume $V_2$ (up to translation),
 it minimizes 
$\mu_1(\Omega_2)$. Accordingly,
in order to minimize $\lambda_1(\Omega)$ for $\Omega = \Omega_1 \times \Omega_2 \in \mathcal{B}(V_1,V_2)$, we must have
\[
\Omega_2 = B_2(V_2).
\] 
Also, by Benguria, Linde and Loewe \cite[Thm. 3.7 and Thm 4.2]{BeLiLo12} the ball $B_1(V_1)$ in $\mathbb{R}^{d_1}$ centered at zero  with  volume $V_1$ is the unique minimizer of the first eigenvalue of the
Schroedinger operator with potential $\mu_{1}(\Omega_2)|x|^{2s}$. Thus, the unique minimum of $\lambda_1(\Omega)$ in $ \mathcal{B}(V_1,V_2)$ is attained by
\[
\Omega = B_1(V_1) \times B_2(V_2),
\]
up to a translation of $B_2(V_2)$.
As one could expect, the minimization problem in $\mathcal{B}(V_1,V_2)$ is trivial and not general enough to capture the anisotropic nature of the problem.
\end{rem}
We  then return to the minimization problem 
\eqref{minpb}.
Let $\Omega = \Omega_1 \times \Omega_2 \in \mathcal{A}(V)$. We set
\begin{align*}
\tilde\Omega_j := |\Omega_j|^{-\frac{1}{d_j}}\Omega_j \quad  j =1,2,
\end{align*}
so that 
\[
|\tilde \Omega_j| = 1 \qquad  j=1,2.
\]
Let $B_1$ be the ball in $\mathbb{R}^{d_1}$ centered at zero with $|B_1|=1$ and $B_2$ a  ball in 
$\mathbb{R}^{d_2}$ with $|B_2|=1$. Then
\[
\mu_1(\Omega_2) = \mu_1(\tilde \Omega_2)|\Omega_2|^{-\frac{2}{d_2}} = \mu_1(\tilde \Omega_2)V^{-\frac{2}{d_2}}|\Omega_1|^{\frac{2}{d_2}}  \geq  \mu_1(B_2)V^{-\frac{2}{d_2}}|\Omega_1|^{\frac{2}{d_2}}.
\]
By equality \eqref{mml1} we have
\begin{align*}
&\lambda_1(\Omega) \\
&= \inf_{f \in H^1_0(\Omega_1)\setminus \{0\}}\frac{\int_{\Omega_1}|\nabla_{x_1} f|^2\,dx_1+\mu_1(\Omega_2)\int_{\Omega_1}|x_1|^{2s}f^2\,dx_1}{\int_{\Omega_1}f^2\,dx_1}\\
&=  \inf_{f \in H^1_0(\tilde \Omega_1)\setminus \{0\}}\Bigg( |\Omega_1|^{-\frac{2}{d_1}}\frac{\int_{\tilde \Omega_1}|\nabla_{x_1} f|^2\,dx_1}{\int_{\tilde \Omega_1}f^2\,dx_1} + \mu_1(\tilde \Omega_2)V^{-\frac{2}{d_2}}|\Omega_1|^{\frac{2}{d_2}+\frac{2s}{d_1}}\frac{\int_{\tilde \Omega_1}|x_1|^{2s}f^2\,dx_1}{\int_{\tilde \Omega_1}f^2\,dx_1}\Bigg)\\
&\geq  \inf_{f \in H^1_0(\tilde \Omega_1)\setminus \{0\}}\Bigg( |\Omega_1|^{-\frac{2}{d_1}}\frac{\int_{\tilde \Omega_1}|\nabla_{x_1} f|^2\,dx_1}{\int_{\tilde \Omega_1}f^2\,dx_1} + \mu_1(B_2)V^{-\frac{2}{d_2}}|\Omega_1|^{\frac{2}{d_2}+\frac{2s}{d_1}}\frac{\int_{\tilde \Omega_1}|x_1|^{2s}f^2\,dx_1}{\int_{\tilde \Omega_1}f^2\,dx_1}\Bigg)\\
&=|\Omega_1|^{-\frac{2}{d_1}}\inf_{f \in H^1_0(\tilde \Omega_1)\setminus \{0\}}\Bigg(\frac{\int_{\tilde \Omega_1}|\nabla_{x_1} f|^2\,dx_1}{\int_{\tilde \Omega_1}f^2\,dx_1} + \mu_1(B_2)V^{-\frac{2}{d_2}}|\Omega_1|^{\frac{2}{d_1}+\frac{2}{d_2}+\frac{2s}{d_1}}\frac{\int_{\tilde \Omega_1}|x_1|^{2s}f^2\,dx_1}{\int_{\tilde \Omega_1}f^2\,dx_1}\Bigg)\\
&= |\Omega_1|^{-\frac{2}{d_1}} E_{1}\left( \mu_1(B_2)V^{-\frac{2}{d_2}} |\Omega_1|^{\frac{2}{d_1}+\frac{2}{d_2}+\frac{2s}{d_1}},\tilde \Omega_1\right)\\
&\geq |\Omega_1|^{-\frac{2}{d_1}} E_{1}\left( \mu_1(B_2)V^{-\frac{2}{d_2}} |\Omega_1|^{\frac{2}{d_1}+\frac{2}{d_2}+\frac{2s}{d_1}},B_1\right).
\end{align*}
The previous inequality gives a lower bound for the first Grushin eigenvalue   in the case of cartesian product domains. We note that the lower bound is attained if and only if $\Omega$ is the product 
of two balls, the first one being centered at zero. Thus, in order to minimize the first eigenvalue 
$\lambda_1(\Omega)$ for $\Omega \in\mathcal{A}(V)$ we 
need to find, if exists,  the volume $|\Omega_1|$ which minimizes the quantity
\[
|\Omega_1|^{-\frac{2}{d_1}} E_{1}\left( \mu_1(B_2) V^{-\frac{2}{d_2}}|\Omega_1|^{\frac{2}{d_1}+\frac{2}{d_2}+\frac{2s}{d_1}},B_1\right).
\]
In other words, we need to find the minimizing $t \in (0,+\infty)$ of the function
\begin{equation}\label{minfunt}
t^{-\frac{2}{d_1}} E_{1}\left( \mu_1(B_2) V^{-\frac{2}{d_2}}t^{\frac{2}{d_1}+\frac{2}{d_2}+\frac{2s}{d_1}},B_1\right).
\end{equation}
For the sake of simplicity  and clarity in the computations, by using the substitution
\begin{equation}\label{subsig}
\sigma = \sigma(t) = \mu_1(B_2) V^{-\frac{2}{d_2}}t^{\frac{2}{d_1}+\frac{2}{d_2}+\frac{2s}{d_1}} \qquad t >0,
\end{equation}
we transform  the minimization problem for $\lambda_1(\Omega)$  into studying the minimizers of 
\begin{equation}\label{Fdef}
F(\sigma) := 
 \sigma^{-\frac{d_2}{d_1+(1+s)d_2}}E_{1}\left(\sigma,B_1\right) \qquad 
\sigma \in (0,+\infty).
\end{equation}
We are able to prove the following concerning the existence of a minimum for $F$.
\begin{prop}\label{minF}
The function $F$ admits a minimum in $(0,+\infty)$.
\end{prop}
\begin{proof}
 Since the first Schr\"odinger eigenvalue $E_1(\sigma,B_1)$ is simple for all $\sigma \in (0,+\infty)$,  by classical analytic perturbation theory (see Rellich \cite{rellich37} and Nagy 
\cite{nagy}) it can be easily seen that $F$ is an analytic function of 
$\sigma \in (0,+\infty)$, and then in particular it is smooth.  A more up do date formulation of abstract perturbation results can be found for example in Lamberti and Lanza de Cristoforis \cite[Thm. 2.27]{lamberti-lanza}. 
Moreover, we claim that
\begin{itemize}
\item[i)] $\lim_{\sigma \to 0^+} F(\sigma) = +\infty$,
\item[ii)] $\lim_{\sigma \to +\infty} F(\sigma) = +\infty$.
\end{itemize}
Statements i), ii) would immediately imply the validity of the lemma. Statement i)   holds because $E_{1}\left( \sigma,B_1\right)$ converges to  the first eigenvalue of the Dirichlet Laplacian in $B_1$, which is strictly positive, when $\sigma$ tends to zero. 

 Next, we consider statement ii). Let $f_\sigma$ be an eigenfunction corresponding to $E_1\left( \sigma,B_1\right)$.  We still denote by $f_\sigma$ its extension by zero in $\mathbb{R}^{d_1}$.  We note that
 \begin{align}\label{minF:ineq}
 E_1\big( \sigma,B_1\big) &=
 \int_{B_1}|\nabla_{x_1} f_\sigma|^2\,dx_1 +  \sigma\int_{B_1}|x_1|^{2s}f_\sigma^2\,dx_1\\ \nonumber
 &=  \int_{\mathbb{R}^{d_1}}|\nabla_{x_1} f_\sigma|^2\,dx_1 +   \sigma\int_{\mathbb{R}^{d_1}}|x_1|^{2s}f_\sigma^2\,dx_1\\ \nonumber
 &\geq E_1\left(\sigma,\mathbb{R}^{d_1}\right)\\\nonumber
 &=\sigma^{\frac{1}{1+s}}E_1\left(1,\mathbb{R}^{d_1}\right).
 \end{align}
We have denoted by $E_1\left(\sigma,\mathbb{R}^{d_1}\right)$ the first eigenvalue of the Schr\"odinger operator $-\Delta_{x_1}+\sigma|x_1|^{2s}$ in $\mathbb{R}^{d_1}$. The inequality in \eqref{minF:ineq} holds because also 
$E_1\left(\sigma,\mathbb{R}^{d_1}\right)$ can be variationally characterized as
\begin{equation}\label{eq:mmpsR}
E_1\left(\sigma,\mathbb{R}^{d_1}\right)= \min_{f \in H^1_0(\mathbb{R}^{d_1})\setminus \{0\}}\frac{\int_{\mathbb{R}^{d_1}}|\nabla_{x_1}f|^2+\sigma|x_1|^{2s}f^2\,dx_1}{\int_{\mathbb{R}^{d_1}}f^2\,dx_1}
\end{equation}
and $H^1_0(\Omega_1) \subseteq H^1_0(\mathbb{R}^{d_1})$. Moreover, the last equality in \eqref{minF:ineq} follows by a simple rescaling argument. 
 Accordingly
\begin{align*}
 F(\sigma)=\sigma^{-\frac{d_2}{d_1+(1+s)d_2}}E_{1}\left(\sigma,B_1\right)
\geq  \sigma^{-\frac{d_2}{d_1+(1+s)d_2}+\frac{1}{1+s}}E_1\left(1,\mathbb{R}^{d_1}\right).
\end{align*}
Since
\[
-\frac{d_2}{d_1+(1+s)d_2}+\frac{1}{1+s} = \frac{d_1}{(d_1+(1+s)d_2)(1+s)} >0
\]
statement ii) holds.
\end{proof}
By the previous discussion and by Proposition \ref{minF}, we deduce that the first Grushin eigenvalue admits at least a minimum in the class of cartesian product domains. Namely, we have the following.
\begin{prop}\label{prop:min}
Let $V>0$. There exists $\Omega^* =\Omega^*_1 \times \Omega^*_2 \in \mathcal{A}(V)$ such that 
\[
\lambda_1(\Omega^*) \leq \lambda_1(\Omega) \qquad \forall\, \Omega \in \mathcal{A}(V).
\]
Moreover, $\Omega^*_1 \subseteq \mathbb{R}^{d_1}$, $\Omega^*_2\subseteq \mathbb{R}^{d_2}$ are two balls, the first one being centered in zero.  
\end{prop}
Our next step is to show that such a minimum is unique (up to translation of $\Omega_2$).  To this aim,
we need to develop some preliminary results. 
\subsection{The Schr\"odinger eigenvalue problem in the ball}
In this section we prove a differential inequality involving the second derivative of the first Schr\"odinger 
eigenvalue of \eqref{eigsplit2} in a ball with respect to the coupling constant $\mu$.

Let $R>0$, $\mu >0$. For the sake of brevity we set
\[
E_1(\mu,R) := E_1\big(\mu,B(0,R)\big), 
\] 
where $E_1\big(\mu,B(0,R)\big)$  is the first eigenvalue of  problem \eqref{eigsplit2}  in 
$B(0,R) = \{x_1 \in \mathbb{R}^{d_1} : |x_1|<R\}$. By using spherical coordinates, since the first eigenfunction is radial, problem \eqref{eigsplit2} for $E_1(\mu,R)$ can be written as
\begin{equation}\label{RadScr}
\begin{cases}
-v''-\frac{d_1-1}{r}v' + \mu r^{2s}v  = E_1(\mu,R) v \qquad \mbox{ on } (0,R),\\
v(R)=0.
\end{cases}
\end{equation}
Let  $v$  be the unique non-negative solution of 
\eqref{RadScr} normalized in $L^2\big((0,R), r^{d_1-1}dr\big)$.
\begin{rem}
As already noted in the proof of Proposition \ref{minF}, since $E_1(\mu,R)$ is simple, classical analytic perturbation theory implies that $E_1(\mu,R)$ and its corresponding eigenfunction $v$ depend analytically upon $\mu >0$. 
Accordingly, the computations performed in this section involving the derivatives of 
 $E_1(\mu,R)$ and  $v$ with respect to $\mu$ are justified.
\end{rem}
As a first step, we need some integral identities.
\begin{lem}\label{intid}
Let $R>0$, $\mu >0$. Let  $v$  be the unique non-negative solution of 
\eqref{RadScr} normalized in $L^2\big((0,R), r^{d_1-1}dr\big)$. 
\begin{align}\label{intid1}
&\int_0^R (v')^2r^{d_1-1}\,dr+\mu\int_0^R r^{2s+d_1-1}v^2\,dr = E_1(\mu,R),\\\label{intid2}
&\int_0^R(v')^2r^{d_1-1}\,dr-\frac{R^{d_1}}{2}(v')^2(R)=s\mu\int_0^Rr^{2s+d_1-1}v^2\,dr,\\\label{intid3}
&\dot E_1(\mu, R) := \frac{\partial }{\partial \mu} E_1(\mu, R) = \int_0^Rr^{2s+d_1-1}v^2\,dr.
\end{align}
\end{lem}
\begin{proof}
In order to prove \eqref{intid1} it suffices to multiply the equation \eqref{RadScr} by $r^{d_1-1}v$ 
and integrate by parts. The identity  \eqref{intid2} follows by multiplying 
\eqref{RadScr} by $r^{d_1}v'$, integrating by parts  and  using \eqref{intid1}. Finally,
 equality  \eqref{intid3} follows by classical abstract results in perturbation theory (see, {\it e.g.}, Lamberti and Lanza de Cristoforis \cite[Theorem 2.30]{lamberti-lanza}) or, with the notation of quantum mechanics, 
 by the Hellmann--Feynman Theorem (see Feynman \cite{Fe39}).
  \end{proof}

By using the identities \eqref{intid1}, \eqref{intid2} and \eqref{intid3} of the previous proposition,  we can recover the following differential identity for 
$E_1(\mu,R)$:
\begin{equation}\label{diffid1}
E_1(\mu,R) - \frac{R^{d_1}}{2}(v')^2(R) = \mu(1+s) \dot E_1(\mu, R).
\end{equation}
  We set 
\[
\dot v := \frac{\partial}{\partial \mu}v, \qquad \ddot E_1(\mu,R) := \frac{\partial^2}{\partial \mu^2}E_1(\mu,R).
\]
 Then by taking the derivative  of \eqref{diffid1}  with respect to $\mu$ we get
\begin{equation}\label{diffid2}
-R^{d_1} v'(R)\dot v'(R) = s \dot E_1(\mu,R) + \mu(1+s)\ddot E_1(\mu, R).
\end{equation}
Our aim is to understand the sign of the right hand side of the previous equality, and we have the following.
\begin{prop}\label{prop:pos}
Let $R>0$, $\mu >0$. Then
\begin{equation*}
 s \dot E_1(\mu,R) + \mu(1+s)\ddot E_1(\mu, R) \geq 0.
\end{equation*}
\end{prop}
\begin{proof}
 Let  $v$  be the unique non-negative solution of 
problem \eqref{RadScr} normalized in $L^2\big((0,R), r^{d_1-1}dr\big)$.
 We note that, since 
$v$ is positive and $v(R) =0$, then
\[
v'(R) \leq 0, \qquad \dot v(R)=0.
\]
Moreover, by taking the $\mu$-derivative of the normalization condition, we deduce that 
\[
\int_0^R r^{d_1-1}v\dot v\,dr =0
\] 
and accordingly $\dot v$ changes sign at least once. We will show that $\dot v$ changes sign only once in $(0,R)$, and 
 that $\dot v' (R) \geq 0$. 
Differentiating \eqref{RadScr} with respect to $\mu$ we obtain
\begin{equation}\label{RadScr2}
-\dot v''-\frac{d_1-1}{r}\dot v'+\mu r^{2s}\dot v + r^{2s}v = \dot E_1(\mu,R)v+E_1(\mu,R)\dot v
\quad \mbox{  in } (0,R).
\end{equation}
By equations \eqref{RadScr} and \eqref{RadScr2} one gets that
\[
\left(r^{d_1-1}(\dot v'v-\dot v v')\right)'=\left(r^{2s-d_1-1}-\dot E_1(\mu,R)r^{d_1-1}\right)v^2 \quad \mbox{ in } (0,R),
\]
that is 
\[
r^{d_1-1}(\dot v'v-\dot v v') = \int_{0}^r\left(t^{2s}-\dot E_1(\mu,R)\right)t^{d_1-1}v^2(t)\,dt
\quad \mbox{   in } (0,R)
\]
We set
\[
\eta(r) := \int_{0}^r\left(t^{2s}- \dot E_1(\mu,R)\right)t^{d_1-1}v^2(t)\,dt  \quad \forall r \in [0,R].
\]
As one can immediately realize, $\eta(0)=\eta(R)=0$. Moreover, since $t^{d_1-1}v^2(t)$ is non-negative for $t \in (0,R)$ and $(t^{2s}-\dot E_1(\mu,R))$ is strictly increasing for $t \in (0,R)$ and negative at $t=0$, then $\eta$ has exactly one critical point and 
\[
\eta(r) <0 \qquad \forall r \in (0,R). 
\] 
This implies that $\frac{\dot v}{v}$ is strictly decreasing in $(0,R)$ since
\[
\left(\frac{\dot v}{v}\right)'=\frac{\dot v'v-\dot v v'}{v^2} <0 \qquad \mbox{ in } (0,R),
\]
and accordingly $\dot v$ can change sign only once and  $\dot v'(R) \geq 0$. We have then proved that 
\begin{equation*}
- R^{d_1} v'(R)\dot v'(R) \geq 0,
\end{equation*}
and the statement follows by equality \eqref{diffid2}.
\end{proof}

\subsection{Uniqueness of the minimum}
 We are now ready to prove the uniqueness of the minimum of problem \eqref{minpb} by means of the following.
\begin{prop}\label{prop:uniqminF}
The function $F$ defined in \eqref{Fdef} has a unique minimum in $(0,+\infty)$.
\end{prop}
\begin{proof}
We take the derivative of $F$. Let $\sigma >0$, then
\begin{align*}
F'(\sigma)=  \sigma^{-\frac{d_2}{d_1+(1+s)d_2}-1}
    \Bigg(-\frac{d_2}{d_1+(1+s)d_2} E_{1}\left(\sigma,B_1\right) 
+\sigma \dot E_{1}\left(\sigma,B_1\right)\Bigg).
\end{align*}
By Proposition \ref{minF}, $F'$ has at least a zero, {\it i.e.} a critical point of $F$.  Accordingly, let $\sigma^*>0$ be a critical point of $F$. Computing the second derivative of $F$ in $\sigma^*$ one gets
\begin{align*}
F''(&\sigma^*)\\
=&\,\,  (\sigma^*)^{-\frac{d_2}{d_1+(1+s)d_2}-1}\\
& \qquad\Bigg(-\frac{d_2}{d_1+(1+s)d_2} \dot E_{1}\left(\sigma^*,B_1\right) 
+ \dot E_{1}\left(\sigma^*,B_1\right)+\sigma^* \ddot E_{1}\left(\sigma^*,B_1\right)\Bigg)\\
=&\,\,  (\sigma^*)^{-\frac{d_2}{d_1+(1+s)d_2}-1}\\
& \quad\Bigg(\frac{d_1}{(d_1+(1+s)d_2)(1+s)} \dot E_{1}\left(\sigma^*,B_1\right) 
+ \frac{s}{1+s}\dot E_{1}\left(\sigma^*,B_1\right)+\sigma^* \ddot E_{1}\left(\sigma^*,B_1\right)\Bigg).
\end{align*}
By identity \eqref{intid3} on the $\mu$-derivative of  $E_1(\mu,R)$, we have that
\[
\dot E_1\left( \sigma^*,B_1\right) > 0.
\]
Moreover, by Proposition \ref{prop:pos},
\begin{align*}
\frac{s}{1+s}\dot E_{1}\left(\sigma^*,B_1\right)+\sigma^* \ddot E_{1}\left(\sigma^*,B_1\right) \geq 0.
\end{align*}
Thus
\[
F''(\sigma^*) > 0.
\]
Since $F$ is smooth in $(0,+\infty)$ and all its critical points have positive second derivative, then $F$ has only one critical point which is a minimum.
\end{proof}
By Proposition \ref{prop:min} and Proposition \ref{prop:uniqminF} we can immediately deduce that the first Grushin eigenvalue admits a unique 
minimum.
\begin{thm}\label{thm:uniqmin}
Let $V>0$. There exists a unique set $\Omega^* =\Omega^*_1 \times \Omega^*_2$ in  $\mathcal{A}(V)$ (up to translations in $\mathbb{R}^{d_2}$) such that 
\[
\lambda_1(\Omega^*) \leq \lambda_1(\Omega) \qquad \forall\, \Omega \in \mathcal{A}(V).
\]
Moreover, $\Omega^*_1 \subseteq \mathbb{R}^{d_1}$, $\Omega^*_2\subseteq \mathbb{R}^{d_2}$ are two balls, the first one being centered in zero.  
\end{thm}
\subsection{Localizing the minimum}\label{sec:locmin}
In this section we obtain some information on the localization of the minimum we proved to be unique in the last section. Let $B_1$ be the ball in $\mathbb{R}^{d_1}$ centered in zero with $|B_1|=1$ and $B_2$ a  ball in 
$\mathbb{R}^{d_2}$ with $|B_2|=1$. Moreover, we denote by
$\tau_d$ the volume of the $d-$dimensional unit ball, that is
\[
\tau_d := \frac{\pi^\frac{d}{2}}{\Gamma(1+\frac{d}{2})} {\color{red}.}
\]
 Let $\sigma >0$ and let $ f_\sigma$ be an eigenfunction corresponding to 
 $E_{1}\left( \sigma,B_1\right)$ normalized in $L^2(B_1)$. We write more explicitly  $F'(\sigma)$. Since
\begin{align*}
\frac{d}{d\sigma}  E_{1}\left( \sigma,B_1\right) 
=  \int_{B_1}|x_1|^{2s}f_\sigma^2\,dx_1,
\end{align*}
then
\begin{align*}
&F'(\sigma) \\
&= -  \Bigg(\frac{d_2}{d_1+(1+s)d_2}\sigma^{-\frac{d_2}{d_1+(1+s)d_2}-1}
E_{1}\left(\sigma,B_1\right)
-  \sigma^{-\frac{d_2}{d_1+(1+s)d_2}}\int_{B_1}|x_1|^{2s}f_\sigma^2\,dx_1\Bigg)\\
&=- \Bigg(\frac{d_2}{d_1+(1+s)d_2}\sigma^{-\frac{d_2}{d_1+(1+s)d_2}-1}
\Big(\int_{B_1}|\nabla_{x_1}f_\sigma|^2\,dx_1+\sigma\int_{B_1}|x_1|^{2s}f^2_\sigma \,dx_1\Big)\\
&\hspace{2cm}-  \sigma^{-\frac{d_2}{d_1+(1+s)d_2}}\int_{B_1}|x_1|^{2s}f_\sigma^2\,dx_1\Bigg)\\
&=- \sigma^{-\frac{d_2}{d_1+(1+s)d_2}-1}\frac{d_2}{d_1+(1+s)d_2}
 \Bigg(
 \int_{B_1}|\nabla_{x_1}f_\sigma|^2\,dx_1-\frac{d_1+sd_2}{d_2}\sigma\int_{B_1}|x_1|^{2s}f^2_\sigma \,dx_1\Bigg).
\end{align*}
%
Let $\sigma^*>0$ be the unique minimum point of $F$, then $F'(\sigma^*) = 0$. That is
\begin{equation}\label{eq:crcond}
\int_{B_1}|\nabla_{x_1}f_{\sigma^*}|^2\,dx_1-\frac{d_1+sd_2}{d_2}\sigma^*\int_{B_1}|x_1|^{2s}f^2_{\sigma^*} \,dx_1=0.
\end{equation}
We note that
\[
\int_{B_1}|\nabla_{x_1}  f_{\sigma*}|^2\,dx_1 \geq \mu_1(B_1) \|f_{\sigma*}\|^2_{L^2(B_1)} =  \mu_1(B_1),
\]
being $\mu_1(B_1)$ be the first Dirichlet Laplacian eigenvalue on the ball $B_1 \subseteq \mathbb{R}^{d_1}$, and
\[
 \int_{B_1}|x_1|^{2s} f_{\sigma^*}^2\,dx_1  \leq \tau_{d_1}^{-\frac{2s}{d_1}} \|f_{\sigma^*}\|^2_{L^2(B_1)} =  \tau_{d_1}^{-\frac{2s}{d_1}},
\]
where we have used the fact that $|x_1| \leq \tau_{d_1}^{-\frac{1}{d_1}}$ for all $x_1 \in B_1$.
Thus
\begin{align*}
\int_{B_1}|\nabla_{x_1}f_{\sigma*}|^2\,dx_1&-\frac{d_1+sd_2}{d_2}\sigma^*\int_{B_1}|x_1|^{2s}f^2_{\sigma^*} \,dx_1 \geq  \mu_1(B_1) -\frac{d_1+sd_2}{d_2}\tau_{d_1}^{-\frac{2s}{d_1}}\sigma^*.
\end{align*}
Since for all $\sigma \in (0,+\infty)$ such that 
\[
\sigma <   \tau_{d_1}^{\frac{2s}{d_1}} \frac{d_2}{d_1+sd_2}\mu_1(B_1) 
\]
one has
\[
\mu_1(B_1) -\frac{d_1+sd_2}{d_2}\tau_{d_1}^{-\frac{2s}{d_1}}\sigma>0,
\]
therefore  the unique minimum point of $F$ must satisfy
\[
\sigma^* \geq \tau_{d_1}^{\frac{2s}{d_1}} \frac{d_2}{d_1+sd_2}\mu_1(B_1).
\]
In other words, recalling the substitution \eqref{subsig}, we have proved the following lower bound.
\begin{prop}\label{prop:lb}
Let $V>0$. Let $\Omega^*_1 \subseteq \mathbb{R}^{d_1}$, $\Omega^*_2 \subseteq \mathbb{R}^{d_1}$ be the two 
balls given by Theorem \ref{thm:uniqmin}. Then
\begin{equation}\label{prop:lb1}
|\Omega^*_1| \geq \left(  \tau_{d_1}^{\frac{2s}{d_1}} \frac{d_2}{d_1+sd_2}\frac{\mu_1(B_1)}{\mu_1(B_2)}V^{\frac{2}{d_2}} \right)^{\frac{d_1 d_2}{2(d_1+(1+s)d_2)}}.
\end{equation}
\end{prop}
By the lower bound \eqref{prop:lb1} of the previous proposition it is also possible to provide a lower bound on $\min_{\Omega \in \mathcal{A}(V)}\lambda_1(\Omega)$.
\begin{prop}\label{prop:lbe}
Let $V>0$. Let $\Omega^*_1 \subseteq \mathbb{R}^{d_1}$, $\Omega^*_2 \subseteq \mathbb{R}^{d_1}$ be the two 
balls given by Theorem \ref{thm:uniqmin}. Then
\begin{align}\label{prop:lbe1}
\lambda_1(&\Omega_1^*\times \Omega_2^*)\\ \nonumber
&\geq  \mu_1(B_2)^{\frac{1}{s+1}}V^{-\frac{2}{d_1+(1+s)d_2}}E_1(1,\mathbb R^{d_1}) \left(\tau_{d_1}^{\frac{2s}{d_1}}\frac{d_2}{d_1+sd_2}\frac{\mu_1(B_1)}{\mu_1(B_2)}\right)^{\frac{d_1}{(s+1)(d_1+(1+s)d_2)}}.
\end{align}
\end{prop}
\begin{proof}
Let $\sigma^*$ be the unique minimizer of $F$. By the  inequality (see \eqref{minF:ineq}):

 \[
 E_1(\sigma^*,B_1) \geq (\sigma^*)^{\frac{1}{s+1}}E_1(1,\mathbb R^{d_1}) \qquad \forall \sigma>0.
 \]
 Then, recalling the substitutions made in \eqref{minfunt}, \eqref{subsig}, \eqref{Fdef} and 
using the lower bound \eqref{prop:lb1}, we have that  the lower bound \eqref{prop:lbe1} holds and the statement is proved.
\end{proof}


 \subsection{Limits as $s \to 0^+$ and $s \to +\infty$}
In this section we study the behavior of the minimization problem when the parameter $s$ tends either to $0$ or to $+\infty$. We use the notation introduced in Section \ref{sec:minprod}.  In particular, 
 $B_1$ denotes the ball in $\mathbb{R}^{d_1}$ centered in zero with $|B_1|=1$ and $B_2$ a  ball in 
$\mathbb{R}^{d_2}$ with $|B_2|=1$.
Let $V>0$.
It will be convenient to introduce a new variable $\tilde\sigma$, defined as follows.
\[
\tilde \sigma = \tilde\sigma(t) =\mu_1(B_2)t^{\frac{2}{d_1}+\frac{2}{d_2}}V^{-\frac{2}{d_2}} \qquad \forall t>0.
\]
Then, if $\sigma$ is the variable defined in \eqref{subsig}, we have
\[
\sigma(t)=\tilde \sigma(t) t^{\frac{2s}{d_1}} \qquad \forall t>0.
\]
The introduction of the variable $\tilde\sigma$ is motivated by the fact that, in this section, we will need to keep the dependence of the coupling constant $\sigma$ explicit on $s$, since we are studying the behavior of the first eigenvalue when $s$  tends to $0$ or to $+\infty$. 
We then set
\begin{equation}\label{Gs}
G_s(t):= t^{-\frac{2}{d_1}}E_1\left(\tilde\sigma(t) t^{\frac{2s}{d_1}},B_1\right) \qquad \forall t>0
\end{equation}
where we recall that for $\mu >0$, $E_1(\mu,\Omega_1)$  denotes the first eigenvalue of problem \eqref{eigsplit2} set in 
$\Omega_1 \subseteq \mathbb{R}^{d_1}$.
We recall that, as noted in Section \ref{sec:minprod} (see in particular \eqref{minfunt}), the unique minimal point of $G_s$ represents the volume of the 
ball in $\mathbb{R}^{d_1}$ of the minimal set. 
We now start to  consider here the limit as $s\rightarrow 0^+$.
\begin{prop}
Let $G_s$ be the function defined in $\eqref{Gs}$. Let
\[
G_0(t):=t^{-\frac{2}{d_1}}(\mu_1(B_1)+\tilde\sigma(t))=t^{-\frac{2}{d_1}}\mu_1(B_1)+t^{\frac{2}{d_2}}V^{-\frac{2}{d_2}}\mu_1(B_2) \quad \forall t >0.
\]
Then
\[
\lim_{s\to0^+}G_s(t) = G_0(t) \qquad \forall t>0.
\]
\end{prop}
\begin{proof}
We start by noting that for all $f\in H^1_0(B_1)$, $\|f\|_{L^2(B_1)}=1$, and $\sigma \in (0,+\infty)$:
\begin{equation}\label{up_s_0}
\int_{B_1}|\nabla_{x_1} f|^2\,dx_1+\sigma\int_{B_1}|x_1|^{2s}f^2\,dx_1\leq \int_{B_1}|\nabla_{x_1} f|^2\,dx_1+\sigma \tau_{d_1}^{-\frac{2s}{d_1}}.
\end{equation}
 Then, taking the minimum over $f\in H^1_0(B_1)$ with $\|f\|_{L^2(B_1)}=1$ to both sides of \eqref{up_s_0}, we immediately obtain
\begin{equation}\label{up_s_1}
E_1(\sigma,B_1)\leq\mu_1(B_1)+\sigma \tau_{d_1}^{-\frac{2s}{d_1}}.
\end{equation}
On the other hand, if $f_{\sigma}$ is the unique (up to sign) eigenfunction associated with $E_1(\sigma,B_1)$ satisfying $\|f_{\sigma}\|_{L^2(B_1)}=1$, then if $B(0,\varepsilon)$ is the ball in $\mathbb{R}^{d_1}$ centered at zero and of radius $\varepsilon\in(0,\tau_{d_1}^{-1/d_1})$
\begin{align}\label{eqE1est}
E_1(\sigma,B_1)&=\int_{B_1}|\nabla_{x_1} f_{\sigma}|^2\,dx_1+\sigma\int_{B_1}|x_1|^{2s}f_{\sigma}^2\,dx_1\\\nonumber
&\geq \int_{B_1\setminus B(0,\varepsilon)}|\nabla_{x_1} f_{\sigma}|^2\,dx_1+\sigma\varepsilon^{2s}\int_{B_1\setminus B(0,\varepsilon)}f_{\sigma}^2\,dx_1\\ \nonumber
&\geq
\int_{B_1\setminus B(0,\varepsilon)}f_{\sigma}^2 \,dx_1
\left(\min_{\substack{ f\in \tilde H^1_0(B_1\setminus B(0,\varepsilon))\\ f \neq 0}}\frac{\int_{B_1\setminus B(0,\varepsilon)}|\nabla_{x_1} f|^2\,dx_1}{\int_{B\setminus B(0,\varepsilon)}f^2\,dx_1}+\sigma\varepsilon^{2s}\right)\\\nonumber
&=\left(\tilde\mu_1(\varepsilon)+\sigma\varepsilon^{2s}\right)\int_{B_1\setminus B(0,\varepsilon)}f_{\sigma}^2 \,dx_1,
\end{align}
where  $\tilde H^1_0(B_1\setminus B(0,\varepsilon))$ denotes the closure of $C^{\infty}_c(B_1)$ in $H^1(B_1\setminus B(0,\varepsilon))$, and $\tilde \mu_1(\varepsilon)$ is the first eigenvalue of the Laplacian in 
$B_1 \setminus B(0,\varepsilon)$ with Dirichlet boundary conditions on $\partial B_1$ and Neumann boundary 
conditions on $\partial B(0,\varepsilon)$. It is well-known that $\tilde\mu_1(\varepsilon)\rightarrow\mu_1(B_1)$ as $\varepsilon\rightarrow 0^+$ (see, {\it e.g.}, Lanza de Cristoforis \cite{La12} and references therein). Moreover
\begin{equation}\label{eqfsb}
1=\int_{B_1}f_{\sigma}^2\,dx_1=\int_{B_1\setminus B(0,\varepsilon)}f_{\sigma}^2\,dx_1+\int_{B(0,\varepsilon)}f_{\sigma}^2\,dx_1.
\end{equation}
If $d_1>2$, we know from the H\"older inequality and the Sobolev inequality in the supercritical case (see, {\it e.g.}
Evans \cite[\S 5.6.3, Thm. 6 (i)]{Ev10}) and from \eqref{up_s_1} that
\[
\int_{B(0,\varepsilon)}f_{\sigma}^2\,dx_1\leq C\varepsilon^2\|\nabla f_{\sigma}\|_{L^2(B_1)}^2\leq C\varepsilon^2 E_1(\sigma,B_1)\leq C\varepsilon^2 \left(\mu_1(B_1)+\sigma \tau_{d_1}^{-\frac{2s}{d_1}}\right).
\]
Moreover, if $d_1=1$, the subcritical Sobolev inequality  (see, {\it e.g.}
Evans \cite[\S 5.6.3, Thm. 6 (ii)]{Ev10}) and \eqref{up_s_1}  imply that
\[
\int_{B(0,\varepsilon)}f_{\sigma}^2\,dx_1\leq C\varepsilon \|\nabla f_{\sigma}\|_{L^2(B_1)}^2\leq C\varepsilon E_1(\sigma,B_1)\leq C\varepsilon \left(\mu_1(B_1)+\sigma \tau_{d_1}^{-\frac{2s}{d_1}}\right).
\]
Finally, if $d_1=2$, exploiting the critical Sobolev inequality (see Burenkov \cite[\S 4.7, Thm. 15]{Bu98})  together again with \eqref{up_s_1}, 
\begin{align*}
\int_{B(0,\varepsilon)}f_{\sigma}^2\,dx_1&\leq C\varepsilon^2(1+|\log(\varepsilon)|) \|\nabla f_{\sigma}\|_{L^2(B_1)}^2\\
&\leq C\varepsilon^2(1+|\log(\varepsilon)|) E_1(\sigma,B_1)\\
&\leq C\varepsilon^2(1+|\log(\varepsilon)|) \left(\mu_1(B_1)+\sigma \tau_{d_1}^{-\frac{2s}{d_1}}\right).
\end{align*}
 See also  Colbois 
and Provenzano \cite[Appendix B]{CoPr18} where the above inequalities are derived with all the details.
In all the cases, the constant $C$ depends only on $d_1$ (in general it depends on the domain, which in this case is $B_1$).
Thus, by the above inequalities, and by \eqref{eqE1est} and \eqref{eqfsb}  one has that  for all $\sigma\in(0,+\infty)$, $s\in(0,+\infty)$ and $\varepsilon\in(0,\tau_{d_1}^{-1/d_1})$ 
\[
E_1(\sigma,B_1)\geq \left(\tilde\mu_1(\varepsilon)+\sigma\varepsilon^{2s}\right)\left(1-C\omega(\varepsilon)\left(\mu_1(B_1)+\sigma \tau_{d_1}^{-\frac{2s}{d_1}}\right)\right),
\]
for some continuous function $w:(0,\tau_{d_1}^{-1/d_1}) \to (0,+\infty)$ such that 
\[
\lim_{\varepsilon\rightarrow 0^+}w(\varepsilon)=0.
\]
Therefore, for all $t\in(0,+\infty)$, $s\in(0,+\infty)$ and $\varepsilon\in(0,\tau_{d_1}^{-1/d_1})$ we have
\begin{align*}
&G_s(t) \geq t^{-\frac{2}{d_1}}\left(\tilde\mu_1(\varepsilon)+\sigma(t)\varepsilon^{2s}\right)\left(1-C\omega(\varepsilon)\left(\mu_1(B_1)+\sigma(t) \tau_{d_1}^{-\frac{2s}{d_1}}\right)\right),\\
 &G_s(t)\leq t^{-\frac{2}{d_1}}\left(\mu_1(B_1)+\sigma(t) \tau_{d_1}^{-\frac{2s}{d_1}}\right).
\end{align*}
Since $\lim_{s\rightarrow 0^+}\sigma(t)=\tilde\sigma(t)$ we have
\[
\limsup_{s\rightarrow 0^+}G_s(t)\leq G_0(t) \qquad \forall t>0
\]
and, for all $\varepsilon\in(0,\tau_{d_1}^{-1/d_1})$
\[
\liminf_{s\rightarrow 0^+}G_s(t)\geq t^{-\frac{2}{d_1}}\Big(\tilde\mu_1(\varepsilon)+\tilde \sigma(t)\Big)\Big(1-C\omega(\varepsilon)\left(\mu_1(B_1)+\tilde\sigma(t)\right)\Big) \qquad \forall t>0.
\]
and then
\[
\liminf_{s\rightarrow 0^+}G_s(t)\geq G_0(t) \qquad \forall t>0.
\]
We have then proved that
\[
\lim_{s\rightarrow 0^+}G_s(t)=G_0(t) \qquad \forall t>0.
\]
and accordingly the statement follows.
\end{proof}
\begin{rem}
Clearly $\mu_1(B_1)+\tilde\sigma$ is the first eigenvalue of problem \eqref{eigsplit2} with $\mu = \tilde \sigma$ when we set $s=0$. We easily see that $G_0(t)$ is optimized when 
\begin{equation}\label{eq:t*0}
t^*=\left(\frac{d_2\mu_1(B_1)}{d_1\mu_1(B_2)}\right)^{\frac{d_1d_2}{2d}}V^{\frac{d_1}{d}}
\end{equation}
 and, accordingly,
\begin{equation}\label{eq:t*G0}
G_0(t^*)=V^{-\frac{2}{d}}\frac{d}{d_1}\mu_1(B_1)\left(\frac{d_1\mu_1(B_2)}{d_2\mu_1(B_1)}\right)^{\frac{d_2}{d}}.
\end{equation}
The optimum given by \eqref{eq:t*0}, \eqref{eq:t*G0} is the expected one, since as $s \to 0^+$ the problem converges to the Dirichlet problem for the classical Laplacian in cartesian product domains.
Note that the limit as $s\rightarrow 0^+$ of the lower bound \eqref{prop:lb1} on the optimal $t$ minimizing $G_s(t)$ for any $s$ (i.e., the optimal $|\Omega^*_1|$) computed in Proposition \ref{prop:lb} equals  the $t^*$ minimizing $G_0$. Therefore the lower bound in Proposition \ref{prop:lb} is sharp in the limit $s\to0^+$.
\end{rem}

Next, we pass to  consider the limit   as $s\rightarrow+\infty$. We first need a preliminary result on the asymptotic behavior of the first eigenvalue 
\[
E_1(s):= E_1(1,\mathbb{R}^{d_1}) 
\]
 of  Schr\"odinger  operator  $-\Delta_{x_1} + |x_1|^{2s}$ on $\mathbb{R}^{d_1}$ as  $s\rightarrow+\infty$. Next 
 Lemma \ref{lem:sclim} is probably known but we include a detailed proof for the sake of completeness.
\begin{lem}\label{lem:sclim}
Let $B(0,1) \subseteq \mathbb{R}^{d_1}$ be the ball of radius one and centered at the origin. Let $\mu_1(B(0,1))$ be the 
first  eigenvalue of the Dirichlet Laplacian on $B(0,1)$.  Then
\begin{equation}\label{lem:sclim1}
\lim_{s\to+\infty}E_1(s) =\mu_1(B(0,1)) = \tau_{d_1}^{-\frac{2}{d_1}}\mu_1(B_1)
\end{equation}
\end{lem}
\begin{proof}
In order to prove \eqref{lem:sclim1} we provide sharp lower and upper bounds. We begin with the upper bound.
Let $h \in (0,1)$. Let $u_h$ be the first $L^2$-normalized eigenfunction of the Dirichlet Laplacian on $B(0,1-h)$. We shall
still denote by $u_h$ the extension by zero of $u_h$ to $\mathbb{R}^{d_1}$. Clearly, the eigenvalue corresponding 
to $u_h$ is 
\[
\frac{\mu_1(B(0,1))}{(1-h)^2}.
\]
From the min-max principle (see \eqref{eq:mmpsR}), we have that 
\begin{equation}\label{lem:sclim2}
E_1(s) \leq \int_{B(0,1-h)} |\nabla_{x_1}u_h|^2+|x_1|^{2s}u_h^2\,dx_1 \leq \frac{\mu_1(B(0,1))}{(1-h)^2}+ (1-h)^{2s}
\end{equation}
for all $s >0$.
Then
\[
\limsup_{s \to +\infty}E_1(s) \leq \frac{\mu_1(B(0,1))}{(1-h)^2}.
\]
Since $h \in (0,1)$ is arbitrary we also deduce that 
\[
\limsup_{s \to +\infty}E_1(s) \leq  \mu_1(B(0,1)).
\]
Note that by letting $h \to 0$ in \eqref{lem:sclim2}, we also obtain that 
\[
E_1(s) \leq \mu_1(B(0,1)) +1
\]
for all $s>0$.

Now we pass to consider the lower bound. Let $s >0$. Let $f_s$ denote an $L^2$-normalized eigenfunction corresponding to $E_1(s)$.
 Then for all $h \in (0,1)$
\begin{align*}
(1+h)^{2s} \int_{\mathbb{R}^{d_1} \setminus \overline{B(0,1+h)}}f^2_s\,dx_1 &\leq
 \int_{\mathbb{R}^{d_1} \setminus \overline{B(0,1+h)}}|x_1|^{2s}f^2_s\,dx_1\\
 &\leq 
E_1(s) \\
&\leq \mu_1(B(0,1)) +1,
\end{align*}
and then
\[
 \int_{\mathbb{R}^{d_1} \setminus \overline{B(0,1+h)}}f^2_s\,dx_1 \leq \frac{ \mu_1(B(0,1)) +1}{(1+h)^{2s}},
\]
 which in turn implies that  $f_s \to 0$ in $L^2\big(\mathbb{R}^{d_1} \setminus \overline{B(0,1+h)}\big)$ for 
 all $h \in (0,1)$.  Next we take $h,h' \in \mathbb{R}$ with  
 $0 \leq h' \leq h \leq 1$ and $\rho \in C_c^{\infty}(\mathbb{R}^{d_1})$ be such that
 \[
 0 \leq \rho \leq 1, \quad \mathrm{supp}\,\rho \subseteq B(0,1+h), \quad \rho_{|B(0,1+h')} = 1.
 \]
 By the min-max principle for the first eigenvalue of the Dirichlet Laplacian in $B(0,1+h)$ with test 
 function $\rho f_s$ and integrating by parts we have
 \[
 (1+h)^{-2}\mu_1(B(0,1))\int_{\mathbb{R}^{d_1}}\rho^2f^2_s\,dx_1 \leq \int_{\mathbb{R}^{d_1}}|\nabla_{x_1}\rho|^2f^2_s\,dx_1 - \int_{\mathbb{R}^{d_1}}\rho^2f_s\Delta_{x_1} f_s\,dx_1.
 \]
 By the eigenvalue equation $-\Delta_{x_1}f_s = E_1(s)f_s-|x_1|^{2s}f_s$ we can deduce that 
 \begin{align*}
  (1+h)^{-2}\mu_1(B(0,1))\int_{\mathbb{R}^{d_1}} \rho^2f^2_s\,dx_1 \leq\,\,&
  \int_{B(0,1+h) \setminus \overline{B(0,1+h')}}|\nabla_{x_1}\rho|^2f^2_s\,dx_1 \\
  &+ E_1(s)\int_{\mathbb{R}^{d_1}}\rho^2f_s^2\,dx_1 - \int_{\mathbb{R}^{d_1}}|x_1|^{2s}\rho^2f_s^2\,dx_1\\
  \leq &\,\,\int_{B(0,1+h) \setminus \overline{B(0,1+h')}}|\nabla_{x_1}\rho|^2f^2_s\,dx_1  \\
  &+ E_1(s)\int_{\mathbb{R}^{d_1}}\rho^2f_s^2\,dx_1.
 \end{align*}
 Since clearly
 \[
 1-\int_{\mathbb{R}^{d_1} \setminus \overline{B(0,1+h')}}f_s^2\,dx_1 \leq \int_{\mathbb{R}^{d_1}} \rho^2f_s^2\,dx_1 \leq 1,
 \]
 then
 \begin{align*}
  (1+h)^{-2}\mu_1(B(0,1)) \leq &\,\, E_1(s) +  \int_{B(0,1+h) \setminus \overline{B(0,1+h')}}|\nabla_{x_1}\rho|^2f^2_s\,dx_1\\
 &+  (1+h)^{-2}\mu_1(B(0,1))\int_{\mathbb{R}^{d_1} \setminus \overline{B(0,1+h')}}\rho^2f^2_s\,dx_1.
 \end{align*}
 Since $\rho^2$ and $|\nabla_{x_1}\rho|^2$ are uniformly bounded, then both 
 $ \int_{B(0,1+h) \setminus \overline{B(0,1+h')}}|\nabla_{x_1}\rho|^2f^2_s\,dx_1$ and  $\int_{\mathbb{R}^{d_1} \setminus \overline{B(0,1+h')}}\rho^2f^2_s\,dx_1$ converges to zero as $s \to +\infty$. That is we proved the lower bound
 \[
 \liminf_{s \to +\infty}E_1(s) \geq  \frac{\mu_1(B(0,1))}{  (1+h)^{2}}
 \]
 which implies 
  \[
 \liminf_{s \to +\infty}E_1(s) \geq \mu_1(B(0,1)),
 \]
 since $h \in (0,1)$ is arbitrary.
 The first equality of \eqref{lem:sclim2} is proved. The second equality  follows simply by rescaling.
\end{proof}

\begin{prop}\label{prop:sinf}
Let $G_s$ be the function defined in $\eqref{Gs}$. Let
\[
G_\infty(t):=\begin{cases}\vspace{.1cm}
\frac{\mu_1(B_1)}{t^{\frac{2}{d_1}}}\,, & t\in(0,\tau_{d_1}),\\ 
\frac{\mu_1(B_1)}{\tau_{d_1}^{\frac{2}{d_1}}}\,, & t\in[\tau_{d_1},+\infty).
\end{cases}
\]
Then
\[
\lim_{s\to+\infty}G_s(t) = G_\infty(t) \qquad \forall t>0.
\]
\end{prop}
\begin{proof}
We need to distinguish two cases: 
 $t<\tau_{d_1}$ and $t\geq \tau_{d_1}$.

Assume first that  $t<\tau_{d_1} $. We have immediately from the min-max principle  that for all $\sigma >0$
\[
E_1(\sigma,B_1)\geq \mu_1(B_1),
\]
and in particular
\[
\liminf_{s\rightarrow+\infty}E_1(\sigma,B_1)\geq\mu_1(B_1).
\]
Moreover, for any $f\in H^1_0(B_1)$ with $\|f\|_{L^2(B_1)}=1$  
\[
\int_{B_1}|\nabla_{x_1} f|^2\,dx_1+\tilde\sigma(t)\int_{B_1}(t^{\frac{1}{d_1}}|x_1|)^{2s}f^2\,dx_1\leq  \int_{B_1}|\nabla_{x_1} f|^2\,dx_1+ \left(t^{\frac{1}{d_1}}\tau_{d_1}^{-\frac{1}{d_1}}\right)^{2s}\tilde\sigma(t).
\]
Taking the minimum over all $f\in H^1_0(B_1)$ with $\|f\|_{L^2(B_1)}=1$ we obtain
\begin{equation}\label{up_2_0}
E_1\left(\tilde\sigma(t) t^{\frac{2s}{d_1}},B_1\right)\leq\mu_1(B_1)+ \left(t^{\frac{1}{d_1}}\tau_{d_1}^{-\frac{1}{d_1}}\right)^{2s}\tilde\sigma(t).
\end{equation}
Since $t<\tau_{d_1}$ we have that 
\[
\limsup_{s\rightarrow+\infty}E_1\left(\tilde\sigma(t) t^{\frac{2s}{d_1}},B_1\right)\leq\mu_1(B_1).
\]
We have proved that
\[
\lim_{s\rightarrow+\infty}E_1\left(\tilde\sigma(t) t^{\frac{2s}{d_1}},B_1\right)=\mu_1(B_1),
\]
and therefore, by recalling the definition \eqref{Gs}, we have
\[
\lim_{s\rightarrow+\infty}G_s(t)=\frac{\mu_1(B_1)}{t^{\frac{2}{d_1}}}.
\]

Next we pass to consider the case   $ t\geq \tau_{d_1}$. Let $L>t$ be fixed. We denote by $\mu_1\left(B\left(0,L^{-\frac{1}{d_1}}\right)\right)$ the first eigenvalue of the Dirichlet Laplacian on the ball $B\left(0,L^{-\frac{1}{d_1}}\right) \subseteq \mathbb{R}^{d_1}$ centered in zero and of radius $L^{-\frac{1}{d_1}}$. Note that the volume of $B\left(0,L^{-\frac{1}{d_1}}\right)$ is $\frac{\tau_{d_1}}{L}<1$. In particular $\mu_1\left(B\left(0,L^{-\frac{1}{d_1}}\right)\right)=\left(\frac{L}{\tau_{d_1}}\right)^{\frac{2}{d_1}}\mu_1(B_1)$. From the inclusion $H^1_0\left(B\left(0,L^{-\frac{1}{d_1}}\right)\right)\subset H^1_0(B_1)$ (where we understand that any $u\in H^1_0\left(B\left(0,L^{-\frac{1}{d_1}}\right)\right)$ is extended by zero to $B_1$), and from analogous computation of those in \eqref{up_2_0}, we obtain

\begin{align*}
E_1\left(\tilde\sigma(t) t^{\frac{2s}{d_1}},B_1\right)&\leq E_1\left(\tilde\sigma(t) t^{\frac{2s}{d_1}},B\left(0,L^{-\frac{1}{d_1}}\right)\right)\\
&\leq\mu_1\left(B\left(0,L^{-\frac{1}{d_1}}\right)\right)+ \left(t^{\frac{1}{d_1}}L^{-\frac{1}{d_1}}\right)^{2s}\tilde\sigma(t)\\
&=\left(\frac{L}{\tau_{d_1}}\right)^{\frac{2}{d_1}}\mu_1(B_1)+ \left(t^{\frac{1}{d_1}}L^{-\frac{1}{d_1}}\right)^{2s}\tilde\sigma(t).
\end{align*}
Then, for all $L>t$,
\[
\limsup_{s\rightarrow+\infty}E_1\left(\tilde\sigma(t) t^{\frac{2s}{d_1}},B_1\right)\leq \left(\frac{L}{\tau_{d_1}}\right)^{\frac{2}{d_1}}\mu_1(B_1),
\]
and therefore
\[
\limsup_{s\rightarrow+\infty}E_1\left(\tilde\sigma(t) t^{\frac{2s}{d_1}},B_1\right)\leq \left(\frac{t}{\tau_{d_1}}\right)^{\frac{2}{d_1}}\mu_1(B_1),
\]
which in turn implies
\[
\limsup_{s\rightarrow+\infty}G_s(t)\leq\frac{\mu_1(B_1)}{\tau_{d_1}^{\frac{2}{d_1}}}.
\]
In order to prove a lower bound we can proceed as follows.
We simply note that, by \eqref{minF:ineq}, for all $\sigma\in(0,+\infty)$,
\[
E_1(\sigma,B_1)\geq E_1(\sigma,\mathbb R^{d_1})=\sigma^{\frac{1}{s+1}}E_1(1,\mathbb R^{d_1}). 
\]
Now, since $\left(\sigma(t)\right)^\frac{1}{s+1}=\left(\tilde\sigma(t)\right)^\frac{1}{s+1} t^{\frac{2s}{d_1(s+1)}}$, then
 $\lim_{s\to+\infty}(\sigma(t))^{\frac{1}{s+1}}=t^{\frac{2}{d_1}}$. Moreover, by Lemma \ref{lem:sclim}  
 \[
\lim_{s\rightarrow+\infty}E_1(1,\mathbb R^{d_1})=\frac{\mu_1(B_1)}{\tau_{d_1}^{\frac{2}{d_1}}}.
\]
 Thus, we immediately deduce that
\[
\liminf_{s\rightarrow+\infty} G_s(t)\geq\frac{\mu_1(B_1)}{\tau_{d_1}^{\frac{2}{d_1}}}.
\]
This implies, along with the upper bound, that for $t\geq \tau_{d_1}$,
\[
\lim_{s\rightarrow+\infty}G_s(t)=\frac{\mu_1(B_1)}{\tau_{d_1}^{\frac{2}{d_1}}}.
\]
Thus the statement is proved.
\end{proof}
 
 \begin{rem}
 Note that in the limiting case $s\rightarrow+\infty$, we have a continuum of optimal $t$, namely all $t\geq \tau_{d_1}$ minimize $G_{\infty}(t)$. We also note that lower bound on the optimal $t$ provided for any $s$  in Proposition \ref{prop:lb} goes to $\tau_{d_1}$ as $s\rightarrow +\infty$, therefore, also in this sense, that lower bound is sharp.
\end{rem}


\section{Some numerical computations}\label{sec:num}
In this last section we present some numerical computations  in the planar case, that is in the case $d_1=d_2=1$. First, we consider the minimization problem of the first eigenvalue in the class of cartesian product domains (i.e. rectangles of $\mathbb{R}^2$). Then we also numerically compute the first eigenvalue in the case the domain is a ball in $\mathbb{R}^2$ and we make some comparisons with the case of rectangles.  For simplicity we  also set $V=1$, but by a simple scaling argument one can also deduce similar results for the general case $V>0$. Note that in this case 
\[
\tau_{d_1}=2, \qquad \mu_1(B_1)\tau^{-\frac{2}{d_1}}=\frac{\pi^2}{4} \approx 2.467 
\]
The numerical scheme to solve the two decoupled one dimensional problems has been implemented in Phyton  with the help of Gabriele Santin 
({\it FBK-ICT}). 

The figures below show the plot of $G_s(t)$ and the numerical computation of its minimum for some values of $s$.  
We recall that the function $G_s(t)$, defined in \eqref{Gs},  equals $\lambda_1(\Omega_1 \times \Omega_2)$ when $\Omega_1 \subseteq \mathbb{R}^{d_1}$, $\Omega_2 \subseteq \mathbb{R}^{d_2}$ are two balls, 
the first one being centered in zero and with $|\Omega_1||\Omega_2|=1$, $|\Omega_1|=t$. The unique minimum 
point $t^*$ of $G_s(t)$ represents then the volume of $\Omega^*_1$, where $\Omega^*_1$ and $\Omega_2^*$ are 
the balls which realizes the minimum for the first eigenvalue  (see Theorem  \ref{thm:uniqmin}).

The first figure corresponds to the limiting case $s \to 0^+$. As expected the minimum is attained 
at $t=1$, which means, with the notation of Theorem \ref{thm:uniqmin}, $|\Omega_1^*|=|\Omega_2^*|=1$. Indeed, for 
the standard Laplacian in two dimension the minimum over the class of cartesian product domains is attained by a square. The other figures show the numerical computation of the minimum for $s=0.5$, $s=1$, $s=2$, $s=3$, and 
$s=150$.

We note that the numerical computations agree with the lower bounds of Propositions \ref{prop:lb} and \ref{prop:lbe}. Moreover, they also agree with the asymptotic behavior as $s \to +\infty$ computed in Proposition 
\ref{prop:sinf}, since $G_s(t)$ tends to flatten to the value $\frac{\pi^2}{4}$ for $t>\tau_{d_1}=2$ when $s$ increases.
  
\begin{figure}[H]
\centering
\begin{minipage}{.5\textwidth}
  \centering
  \includegraphics[width=\linewidth]{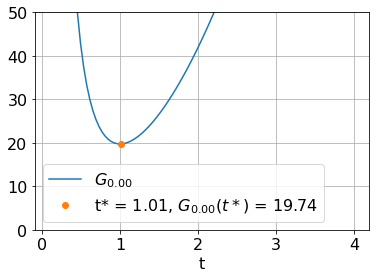}
  \captionof*{figure}{s=0}
\end{minipage}%
\begin{minipage}{.5\textwidth}
  \centering
  \includegraphics[width=\linewidth]{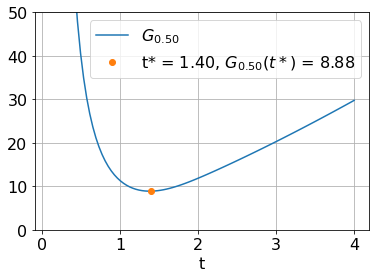}
  \captionof*{figure}{s=0.5}
\end{minipage}
\end{figure}
\begin{figure}[H]
\begin{minipage}{.5\textwidth}
  \centering
  \includegraphics[width=\linewidth]{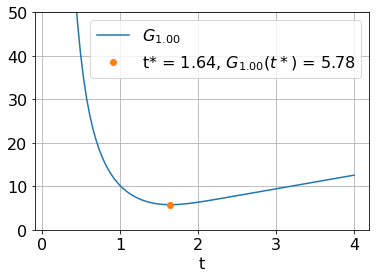}
  \captionof*{figure}{s=1}
\end{minipage}%
\begin{minipage}{.5\textwidth}
  \centering
  \includegraphics[width=\linewidth]{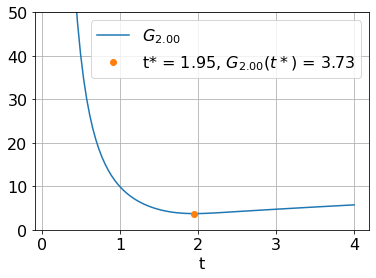}
  \captionof*{figure}{s=2}
\end{minipage}
\end{figure}

\begin{figure}[H]
\begin{minipage}{.5\textwidth}
  \centering
  \includegraphics[width=\linewidth]{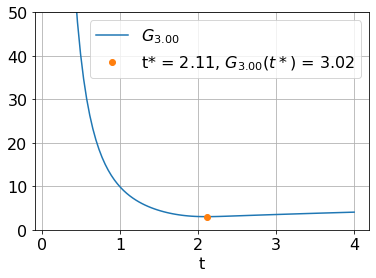}
  \captionof*{figure}{s=3}
\end{minipage}%
\begin{minipage}{.5\textwidth}
  \centering
  \includegraphics[width=\linewidth]{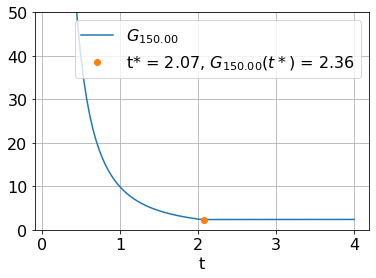}
  \captionof*{figure}{s=150}
\end{minipage}
\end{figure}

We conclude by comparing the eigenvalues on rectangles in $\mathbb R^2$  with those of the disk with the same area, centered at the origin.

\begin{figure}[H]
\centering
  \includegraphics[width=\linewidth]{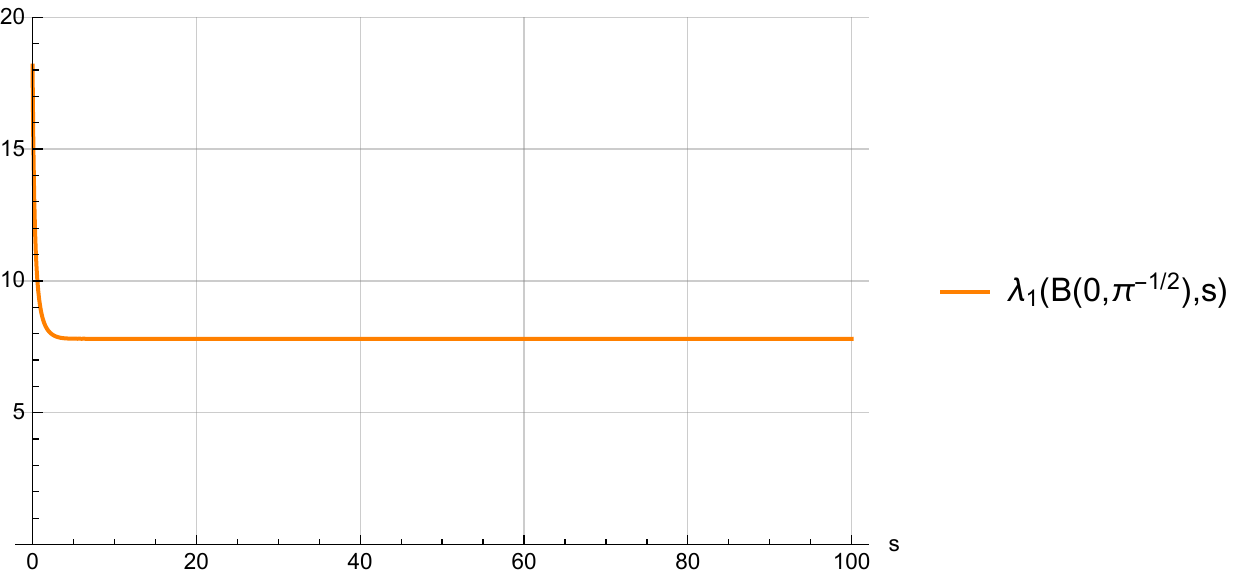}
  \captionof*{figure}{$\lambda_1(B(0,\pi^{-1/2}))$ as a function of $s$.}
\end{figure}
When $s=0$, clearly the disk of unit area (i.e., radius $\pi^{-1/2}$) has lower first eigenvalue than any rectangle of the same area (Faber-Krahn inequality). In fact, for $s=0$
$$
\lambda_1(B(0,\pi^{-1/2}))\approx 18.17 {\rm \ \ \ and\ \ \ }\lambda_1(\Omega^*_1\times\Omega^*_2)\approx 19.74
$$
Already when $s=0.5$, we have
$$
\lambda_1(B(0,\pi^{-1/2}))\approx 10.45  {\rm \ \ \ and\ \ \ }\lambda_1(\Omega^*_1\times\Omega^*_2)\approx 8.88
$$
and, when $s=1$, we have
$$
\lambda_1(B(0,\pi^{-1/2}))\approx 8.90 {\rm \ \ \ and\ \ \ } \lambda_1(\Omega^*_1\times\Omega^*_2)\approx 5.78
$$
We also note that, as $s\rightarrow+\infty$
$$
\lambda_1(B(0,\pi^{-1/2}))\approx 7.75\approx\frac{\pi^3}{4}
$$
while
$$
\lambda_1(\Omega^*_1\times\Omega^*_2)\approx\frac{\pi^2}{4}
$$
The numerics suggest that there exists some $s_0<\frac{1}{2}$ such that the disk of unit area is no more the minimizer among all domains of the same area, and we always find a rectangle doing better.

We have also computed the first eigenvalue on the disk of radius $1$ and of radius $2$ as functions of $s$.

\begin{figure}[H]
\centering
  \includegraphics[width=\linewidth]{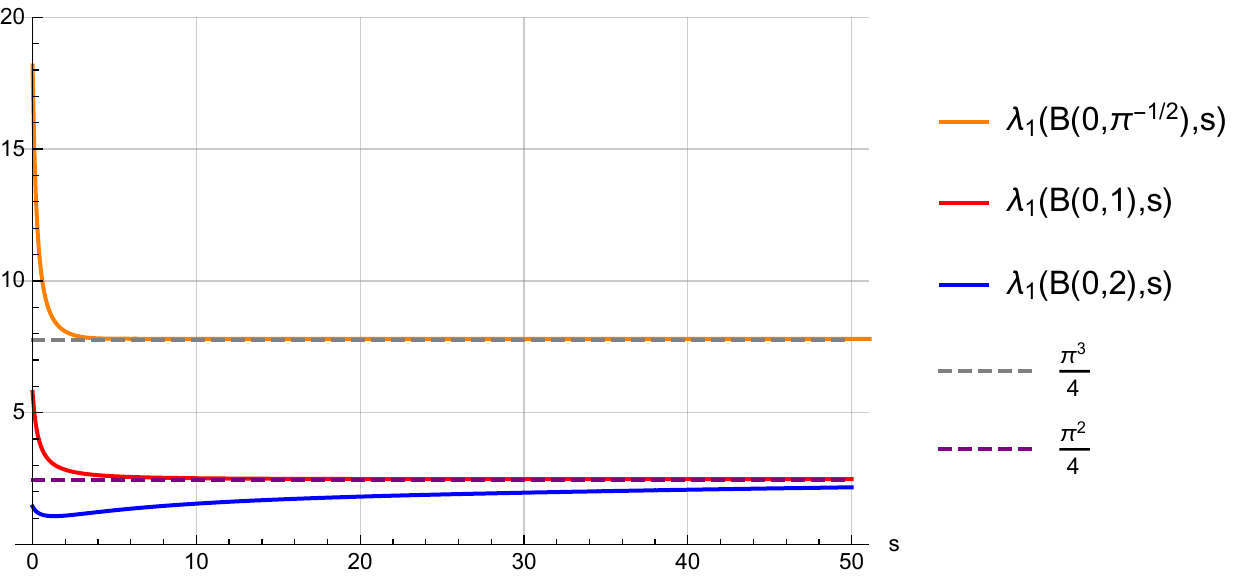}
  \captionof*{figure}{$\lambda_1(B(0,\pi^{-1/2}))$, $\lambda_1(B(0,1))$ and $\lambda_1(B(0,2))$  as a function of $s$.}
\end{figure}

We note that, as $s\rightarrow+\infty$, the first eigenvalue of the disks of radius $1$ and $2$ seems to behave like $\frac{\pi^2}{4}$, exactly as $\lambda_1(\Omega_1\times\Omega_2)$ with $|\Omega_1|\geq 2$ (which means, lenght of the side parallel to the $x_1$-axis greater than $2$).

We note that $\frac{\pi^2}{4}$ is exactly the first Dirichlet eigenvalue of an interval of length $2$. On the other hand, the value $\frac{\pi^3}{4}$, the expected limit of the first eigenvalue of the disk of area $1$ as $s\to+\infty$, is the first Dirichlet eigenvalue of an interval of length $2\pi^{-1/2}$.

It looks like that the behavior of the first eigenvalue of a domain, as $s\to+\infty$, is determined by the length of the longest segment parallel to the $x_1$-axis contained in $\Omega\cap \{|x_1|<1\}$, which is $2$ in the case of $B(0,1)$, $B(0,2)$ and of any rectangle $\Omega_1\times\Omega_2$ with $|\Omega_1|\geq 2$, and is $2\pi^{-1/2}$ for $B(0,\pi^{-1/2})$. We will consider these issues from an analytical point of view in future works. At any rate, we are left with the following

{\bf Question.} Does $\lambda_1(\Omega)\to\lambda_1((0,L))$ as $s\to+\infty$, where $L$ is the lenght of the longest segment parallel to the $x_1$-axis contained in $\Omega\cap \{|x_1|<1\}$, and $\lambda_1((0,L))$ is the first Dirichlet eigenvalue on $(0,L)$?

\section*{Acknowledgements}

The authors are deeply thankful to Gabriele Santin ({\it FBK-ICT}) for the help with the numerical computations. 
The first and third authors acknowledge support of the SNSF project ``Bounds for the Neumann and Steklov eigenvalues of the biharmonic operator'', grant number 200021\_178736.
The first   author is member of the Gruppo Nazionale per l'Analisi
Matematica, la Probabilit\`a e le loro Applicazioni (GNAMPA) of the I\-sti\-tuto Naziona\-le di Alta Matematica (INdAM). The second author is member of the Gruppo Nazionale per le Strutture Algebriche, Geometriche e le loro Applicazioni (GNSAGA) of the Istituto Nazionale di Alta Matematica (INdAM).


\end{document}